\documentclass[10pt,a4paper, reqno]{amsart}
\usepackage{mathrsfs}

\usepackage{amsmath,amsfonts,verbatim}
\usepackage{latexsym}
\usepackage{amssymb,leftidx}
\usepackage{extarrows}
\usepackage{overpic}
\usepackage{color}
\usepackage{epsfig}
\usepackage{subfigure}
\usepackage{tikz}
\usepackage{bbm}
\usepackage{tikz, caption}
\usepackage[font=small,labelfont=bf]{caption}

\usepackage[colorlinks=true,backref=page]{hyperref}
\hypersetup{urlcolor=blue, citecolor=red}

\usepackage{amssymb}
\usepackage{mathrsfs}
\usepackage{amscd}
\usepackage{cite}

\newcommand\R{\mathbb{R}}
\newcommand\C{\mathbb{C}}
\newcommand\Z{\mathbb{Z}}

\newcommand\A{{\bf A}}

\newcommand\LL{\mathcal{L}}

\usepackage{graphicx}
\usepackage{float}

\numberwithin{equation}{section}
\newtheorem{proposition}{Proposition}[section]
\newtheorem{definition}{Definition}[section]

\newtheorem{theorem}{Theorem}[section]

\newtheorem{remark}{Remark}[section]

\allowdisplaybreaks
\begin{document}
\title[ Critical magnetic Schr\"odinger operators on flat Euclidean cones ]{Strichartz estimates for Critical magnetic Schr\"odinger operators on flat Euclidean cones}

\author{Xiaofen Gao}
\address{School of Mathematics and Statistics, Zhengzhou University, Zhengzhou, Henan, China, 450001; }
\email{gaoxiaofen@zzu.edu.cn}

\author{Jialu Wang}
\address{School of Mathematics and Statistics, Beijing Institute of Technology, Beijing, China, 100081; }
\email{jialu\_wang@bit.edu.cn}

\author{Chengbin Xu}
\address{School of Mathematics and Statistics, Qinghai Normal University, Xining, Qinghai, China, 810008; }
\email{xcbsph@163.com}

\author{Fang Zhang}
\address{Department of Mathematics, Beijing Institute of Technology, Beijing 100081}
\email{zhangfang@bit.edu.cn}

\begin{abstract} In this paper, we study Schr\"{o}dinger operator $\mathcal{H}_{\mathbf{A}}$ perturbed by critical magnetic potentials on the 2D flat cone $\Sigma = C(\mathbb{S}_\rho^1) = (0, \infty) \times \mathbb{S}_\rho^1$, which is a product cone over the circle $\mathbb{S}_\rho^1 = \mathbb{R}/2\pi \rho \mathbb{Z}$ with radius $\rho > 0$, and equipped with the metric $g = dr^2 + r^2 d\theta^2$. The goal of this work is to establish Strichartz estimates for $\mathcal{H}_{\mathbf{A}}$ in this setting. A key aspect of our approach is the construction of the Schwartz kernel of the resolvent and the spectral measure for Schr\"{o}dinger operator on the flat Euclidean cone $(\Sigma, g)$. The results presented here generalize previous work in \cite{Ford, BFM, FZZ, Zhang1}.
\end{abstract}

\maketitle
\begin{center}
\begin{minipage}{120mm}
\end{minipage}
\end{center}

\section{Introduction}In this paper, we continue our previous work in \cite{GYZZ, GZZ} to investigate the decay estimates for Schr\"{o}dinger operators perturbed by critical magnetic potentials on two-dimensional flat Euclidean cones. As studied in \cite{BFM, CT1, CT2, Ford}, the flat Euclidean cone $ \Sigma = C(\mathbb{S}_\rho^1) = (0, \infty) \times \mathbb{S}_\rho^1 $ is a product cone over the circle $ \mathbb{S}_\rho^1 = \mathbb{R} / 2\pi \rho \mathbb{Z} $, with radius $ \rho > 0 $, and is equipped with the metric $ g = dr^2 + r^2 d\theta^2 $. This space is a generalization of the Euclidean plane $ \mathbb{R}^2 $ and is a special type of metric cone, as concerned in \cite{HL, ZZ1, ZZ2}.

More precisely, in this setting, we study the Strichartz estimates for the Schr\"{o}dinger and wave equations associated with the Schr\"{o}dinger operator $ \mathcal{H}_{\bf A} $. The operator we consider is given by
\begin{equation}\label{oper}
\mathcal{H}_{\bf A} = -\partial_r^2 - \frac{1}{r} \partial_r + \frac{(i\nabla_\theta + {\bf A}(\theta))^2}{r^2}, \qquad (r, \theta) \in \Sigma,
\end{equation}
where ${\bf A}(\theta): C^\infty(\mathbb{S}_\rho^1) \to \mathbb{R}$, and $(i\nabla_\theta + {\bf A}(\theta))^2$ represents the Laplace operator with magnetic potential on $\mathbb{S}_\rho^1$. We focus on its Friedrichs extension in this paper. In particular, when $\Sigma = \mathbb{R}^2$, the operator can be written as
\begin{equation}\label{oper-1}
\mathcal{L}_{\bf A} = \left(i\nabla + \frac{{\bf A}(\hat{x})}{|x|}\right)^2, \qquad x = (x_1, x_2) \in \mathbb{R}^2,
\end{equation}
where $\hat{x} \in \mathbb{S}^1$ and ${\bf A} \in W^{1,\infty}(\mathbb{S}^1, \mathbb{R}^2)$ satisfies the transversal gauge condition
\begin{equation}
{\bf A}(\hat{x}) \cdot \hat{x} = 0, \quad \text{for all} \ \hat{x} \in \mathbb{S}^1.
\end{equation}
This model encompasses an interesting physical Aharonov-Bohm potential \cite{AB}, given by
\begin{equation}\label{AB:po}
{\bf A}_B(\hat{x}) = \alpha \left( -\frac{x_2}{|x|}, \frac{x_1}{|x|} \right), \quad \alpha \in \mathbb{R}.
\end{equation}

We refer the reader to \cite{AB} for the Aharonov-Bohm effect and to \cite{AW} for typical cosmic-string scenarios. Mathematically, the potential in \eqref{oper} is doubly critical due to its scaling invariance and the singularity it possesses, which introduces additional complexity to dispersive estimates. For further studies on dispersive and Strichartz estimates for the Schr\"{o}dinger, wave, and Klein-Gordon equations with critical magnetic potentials in \( \mathbb{R}^2 \), we direct the reader to \cite{FFFP1, FFFP, FZZ, GYZZ}.\vspace{0.2cm}

Strichartz inequalities are well-established for the constant coefficient Schr\"{o}dinger and wave equations on $\mathbb{R}^n$; see \cite{GV95, KT, LS, strichartz} and the references therein. For instance, the Strichartz inequalities for the Schr\"{o}dinger equation assert that for admissible pairs $(q, r)$, defined by
\begin{equation}
(q, r) \in \left\{(q, r) \in [2, \infty] \times [2, \infty] : \frac{2}{q} = n \left( \frac{1}{2} - \frac{1}{r} \right) \right\},
\end{equation}
the solution $u(t, x)$ satisfies the following estimate:
\begin{equation}
\|u(t, x)\|_{L^{q}_t(\mathbb{R}; L_x^r(\mathbb{R}^n))} \leq C \|u_0\|_{L^2(\mathbb{R}^n)},
\end{equation}
where $C$ is a universal constant. Strichartz estimates were first introduced in 1977 by R. Strichartz \cite{strichartz}, who established a type of  priori estimates. Since then, numerous works have been dedicated to this type of priori estimates, which are now commonly referred to as Strichartz estimates, particularly for solutions to the Schr\"{o}dinger and wave equations. These estimates are widely used to derive significant results in the well-posedness theory and nonlinear scattering theory for semilinear Schr\"{o}dinger and wave equations on Euclidean space. For further details, see \cite{GV85,Tbook}, and the references therein.\vspace{0.2cm}

It is natural to ask whether Strichartz inequalities hold for Schr\"{o}dinger operators with variable coefficients, especially in settings on manifolds and for operators with potentials \cite{Tbook}. Due to their important applications in nonlinear theory and partial differential equations, many mathematicians have conducted extensive research and achieved significant results. By developing a representation of the fundamental solution \cite{CT1, CT2}, Ford \cite{Ford} and Blair-Ford-Marzuola \cite{BFM} were pioneers in proving the Strichartz inequalities on the flat cone. However, this method cannot be applied to more general settings, such as metric cones or asymptotically conic manifolds. The Strichartz estimates for the Schr\"{o}dinger and wave equations in these general settings were established by Hassel, Zhang, and Zheng in their series of works \cite{HZ, Zhang, ZZ1, ZZ2}, and their results were subsequently applied to the proof of scattering theory for nonlinear equations. Moreover, the dispersive estimates on metric cones were resolved in the recent work \cite{JZ1} using a modified Hadamard parametrix. Additionally, Jia-Zhang \cite{JZ2} proved the dispersive and Strichartz estimates for Schr\"{o}dinger operators perturbed by critical electromagnetic potentials in dimensions $n \geq 3$. This case is more complicated than the $n = 2$ case due to the explicit eigenfunctions and eigenvalues in $\mathbb{R}^2$.\vspace{0.2cm}

Nevertheless, because of the perturbed critical potentials, the proof of dispersive and Strichartz estimates in \(\mathbb{R}^2\) is not straightforward and has taken several decades to develop. We will not cite all related works, but will only introduce those most relevant to the operators studied in the present paper. Strichartz and time decay estimates for the Schr\"{o}dinger equation were established by Fanelli, Felli, Fontelos, and Primo in \cite{FFFP, FFFP1}, utilizing pseudoconformal invariance. Later, Fanelli, Zhang, and Zheng \cite{FZZ} proved Strichartz estimates based on a Lipschitz-Hankel integral formula. More recently, Yin, Zhang, Zheng, and the first author \cite{GYZZ} established time decay and Strichartz estimates for Klein-Gordon equations via the construction of spectral measures. Following a similar approach to that used in analyzing Schr\"{o}dinger operators without potentials on the 2D flat Euclidean cone \(\Sigma\), Zhang \cite{Zhang1} derived dispersive estimates for Schr\"{o}dinger and wave equations, from which Strichartz estimates follow by Keel and Tao's abstract argument \cite{KT}. Based on these works, Yin and Zhang \cite{YZ} established dispersive and Strichartz inequalities for the Klein-Gordon equation on flat Euclidean cones. Motivated by the aforementioned papers and following the methods in \cite{Zhang1, GYZZ}, we aim to establish Strichartz estimates for Schr\"{o}dinger and wave equations on the 2D flat cone perturbed by critical magnetic potentials as described in \eqref{AB:po}.\vspace{0.2cm}

More precisely, we have the following results
\begin{theorem}\label{thm:stri0}Let $\mathcal{H}_{\bf A}$ be in \eqref{oper}, and let $u(t,x)$ a solution of Schr\"odinger equation
\begin{equation}\label{equ:S'}
\begin{cases}
i\partial_t u+\mathcal{H}_{{\A}} u=0,\qquad (t,x)\in\R\times \Sigma\\
u(0,x)=u_0(x),\quad x\in \Sigma,
\end{cases}
\end{equation}
 Then there exists a constant $C$ such that
\begin{equation}\label{stri-LA}
\|u(t,x)\|_{L^{q}_{t}(\R;L_x^{r}( \Sigma))}\leq C\|u_0\|_{L^2( \Sigma)},
\end{equation}
where
 $(q,r)\in\Lambda_0^S$ defined in  below
\begin{equation}\label{adm}
(q,r)\in\Lambda^S_0:=\Big\{(q,r)\in[2,\infty]\times[2,\infty):\tfrac2q=2\big(\tfrac12-\tfrac1r\big)\Big\}.
\end{equation}
\end{theorem}
To state the our another main result, we need to give the definition of distorted Besov spaces. Let $\phi\in C_0^\infty(\R\backslash\{0\}),$ with $0\leq\phi\leq1$, supp$\,\phi\subset[\frac12,1]$ and
\begin{equation}\label{def-phi}
\sum_{j\in\Z}\phi(2^{-j}\lambda)=1,\quad\phi_j(\lambda):=\phi(2^{-j}\lambda),\quad j\in\Z.
\end{equation}
\begin{definition}[Besov spaces associated with $\mathcal{H}_{{\A}}$]
For $s\in\R$ and $1\leq p,r\leq\infty$, the norm of $\|\cdot\|_{\dot{\mathcal{B}}_{p,r,{\bf A}}^s( \Sigma)}$ is given by
\begin{equation}
\|f\|_{\dot{\mathcal{B}}_{p,r,{\bf A}}^s( \Sigma)}=\Big(\sum_{j
\in\Z}2^{jsr}\|\phi_j(\sqrt{\mathcal{H}_{\bf A}})f\|_{L^p( \Sigma)}^r\Big)^{\frac1r}.
\end{equation}
In particular, for $p=r=2$, we have
\begin{equation}
\|f\|_{\dot{H}^s_{\bf A}( \Sigma)}:=\|\mathcal{H}_{\bf A}^{\frac s2}f\|_{L^2( \Sigma)}=\Big\|\Big(\sum_{j
\in\Z}2^{2js}|\phi_j(\sqrt{\mathcal{H}_{\bf A}})f|^2\Big)^{\frac12})\Big\|_{L^2( \Sigma)}
=\|f\|_{\dot{\mathcal{B}}_{2,2,{\bf A}}^s( \Sigma)}.
\end{equation}
\end{definition}
Our main results about the Strichartz estimates of the wave equation is the following.
\begin{theorem}\label{thm:stri-w}Let $\mathcal{H}_{\bf A}$ be in \eqref{oper}, and let $u(t,x)$ a solution of wave equation
\begin{equation}\label{equ:w'}
\begin{cases}
\partial_{tt} u+\mathcal{H}_{{\A}} u=0,\qquad (t,x)\in\R\times  \Sigma\\
u(0,x)=f(x),\qquad \partial_t u(0,x)= g(x),\qquad x\in  \Sigma.
\end{cases}
\end{equation}
Then there exists a constant $C$ such that
\begin{equation}\label{stri-w}
\|u(t,x)\|_{L^{q}_{t}(\R;L_x^r( \Sigma))}\leq C\Big(\|f\|_{\dot{\mathcal{H}}^s_{\bf A}(\Sigma)}+\|g\|_{\dot{\mathcal{H}}^{s-1}_{\bf A}( \Sigma)}\Big),
\end{equation}
where $s\geq0$ and
\begin{equation}\label{adm-w}
(q,r)\in\Lambda^w_s:=\Big\{(q,r)\in[2,\infty]\times[2,\infty):\tfrac2q\leq \big(\tfrac12-\tfrac1r\big), s=2\big(\tfrac12-\tfrac1r\big)-\tfrac1q\Big\}.
\end{equation}
\end{theorem}
\begin{remark} The Strichartz inequalities presented here include the results of Ford \cite{Ford} and Blair-Ford-Marzuola \cite{BFM}, and extend to Schr\"{o}dinger operators with critical magnetic potentials.
\end{remark}
\begin{remark}
In the future, we will use the results presented in our paper to address several nonlinear problems (such as well-posedness and scattering theory) associated with $\mathcal{H}_{\mathbf{A}}$ in this setting.
\end{remark}

Here, we provide an outline of the proof. As mentioned above, the main idea of this paper is to construct the Schwartz kernel of the resolvent and the spectral measure (see \eqref{sp} below) of the Schr\"{o}dinger operator $\mathcal{H}_{\mathbf{A}}$, perturbed by critical magnetic potentials as given in \eqref{AB:po}. 

The paper is organized as follows: In Section \ref{sec-s}, we prove Theorem \ref{thm:stri0} by constructing the Schr\"{o}dinger propagator $e^{it\mathcal{H}_{\bf A}}$, while Section \ref{sec-w} is devoted to proving Theorem \ref{thm:stri-w}, which relies on the representation of the spectral measure on $\Sigma$.

\section{Strichartz estimates for Schr\"odinger}\label{sec-s}

This section is devoted to showing Theorem \ref{thm:stri0}. Form the abstract method of Keel-Tao \cite{KT} and the energy estimates (obtained from the mass conservation law for Schr\"odinger equation or the unitary property of $e^{it\mathcal{H}_{\A}}$), it suffices to prove the following dispersive estimates.
\begin{theorem}[Dispersive for Schr\"odinger]\label{thm:Sch-dis} Let $\mathcal{H}_{\bf A}$ in \eqref{oper} be the Schr\"odinger operator with magnetic potentials on the flat cone $ \Sigma$, and let $x=(r,\theta)\in \Sigma$ and $y=(r^\prime,\theta^\prime)\in  \Sigma$. Then the dispersive estimate holds
\begin{equation}\label{dis-s}
|e^{-it\mathcal{H}_{\bf A}}(x,y)|\leq C_\rho |t|^{-1},\quad \forall t\in\R,
\end{equation}
where $C_\rho$ is a constant independent of $x,y$.
\end{theorem}

To obtain Theorem \ref{thm:Sch-dis}, the key point is the representation of Schr\"odinger kernel of $e^{it\mathcal{H}_{\bf A}}$, which can be constructed by following the arguments in \cite{GYZZ,YZ}. Indeed, one has
\begin{proposition}[Schr\"odinger kernel]\label{pro-SKEr}Let $x=(r,\theta)\in  \Sigma$ and $y=(r',\theta')\in  \Sigma$, and denote two constants
\begin{equation}\label{def-phi}
\phi_1:=\frac1\rho(\pi-(\theta-\theta')),\quad \phi_2:=\frac1\rho(-\pi-(\theta-\theta')).
\end{equation}
Let $\alpha$ be in \eqref{def-alph} below, and define \begin{equation}\label{def-A}
A_{\alpha,\rho}(\theta,\theta'):=e^{i\int_{\theta'}^\theta A(\theta^*)\, d\theta^*}e^{i\alpha\cdot 2j\rho\pi},\quad j\in\Z,
\end{equation}
and \begin{equation}\label{def-B}
\begin{split}
&B_{\alpha,\rho}(s,\theta,\theta'):=e^{-i\big( (\theta-\theta')\alpha-\int_{\theta'}^\theta A(\theta^*)\, d\theta^*\big)}\times\Big\{\sin(|\alpha|\pi)e^{-|\alpha|s}\\
&+\frac1{4i}e^{-s\alpha}\Big[\frac{\cos\phi_1-e^{-\frac s\rho}+i\sin\phi_1}{\cosh \frac s\rho-\cos\phi_1}e^{i\alpha\pi}-\frac{\cos\phi_2-e^{-\frac s\rho}-i\sin\phi_2}{\cosh \frac s\rho-\cos\phi_2}e^{-i\alpha\pi}\Big]\\
&-\frac1{4i}e^{s\alpha}\Big[\frac{\cos\phi_1-e^{-\frac s\rho}-i\sin\phi_1}{\cosh \frac s\rho-\cos\phi_1}e^{i\alpha\pi}-\frac{\cos\phi_2-e^{-\frac s\rho}+i\sin\phi_2}{\cosh \frac s\rho-\cos\phi_2}e^{-i\alpha\pi}\Big]\Big\}.
\end{split}
\end{equation}
Let $K(t,r,\theta,r',\theta')$ be the kernel of $e^{it\mathcal{H}_{{\A}}}$ with $\mathcal{H}_{{\A}}$ being in \eqref{oper}. Then there holds
\begin{equation}\label{pro-SK}
\begin{split}
K(t,r,\theta,r',\theta')
=&\frac1{4\pi}\frac{e^{-\frac{r^2+r'^2}{4it}} }{it}
\sum_{\{j\in\Z: 0\leq |\theta-\theta'+2j\rho\pi|\leq \pi\}} A_{\alpha,\rho}(\theta,\theta')e^{\frac{rr'}{2it}\cos(\theta-\theta'+2j\rho\pi)}\\
&-\frac{1}{4\pi^2\rho}\frac{e^{-\frac{r^2+r'^2}{4it}} }{it}
 \int_0^\infty e^{-\frac{rr'}{2it}\cosh s} B_{\alpha,\rho}(s,\theta,\theta') \, ds.
\end{split}
\end{equation}
\end{proposition}

\begin{remark} If $\rho=1$ and $\alpha=0$, then $B_{\alpha,\rho}(s,\theta,\theta')$ vanishes. The first term becomes
$$(4\pi it)^{-1}e^{-\frac{|x-y|^2}{4it}}, $$
which consists with the kernel of Schr\"odinger propagator in Euclidean space $\R^2$.
\end{remark}
\begin{remark} Let $\alpha$ denote the total flux along the $\mathbb{S}_\rho^1$
\begin{equation}\label{def-alph}\alpha=\frac1{2\pi\rho}\int_0^{2\pi\rho}A(\theta)\, d\theta.
\end{equation}
Due to the unitarity equivalence of magnetic Schr\"{o}dinger operators for $\alpha$ and $\alpha + 1$, without loss of generality, we assume that $\alpha \in \left( -\frac{1}{\rho}, \frac{1}{\rho} \right)\setminus\{0\}$.
\end{remark}

\subsection{The proof of Proposition \ref{pro-SKEr}} We first write the positive Laplacian $\mathcal{H}_{\bf A}$ perturbed by magnetic potentials on $\Sigma$ as follows
\begin{equation}\label{oper-1}
\mathcal{H}_{\bf A}=-\partial_r^2-\frac1r\partial_r+\frac{{L_{\bf A}}}{r^2},
\end{equation}
where $L_{\bf A}$ is the operator on $\mathbb{S}_\rho^1$ given by
$ L_{\bf A}=(i\nabla_\theta+A(\theta))^2$.
Let $\nu_k=|\frac{k}{\rho}+\alpha|$ and
\begin{equation}
\varphi_k(\theta)=\frac1{\sqrt{2\pi\rho}}e^{-i\big(\theta(\frac k\rho+\alpha)-\int_0^{\theta}A(\theta^*)\, d  \theta^*\big )},\qquad k\in\Z,
\end{equation}
then $\nu_k^2$ and $\varphi_k$ are eigenvalues and eigenfunctions of operator $L_{\bf A}$ such that
\begin{equation}
L_{\bf A}\varphi_k(\theta)=\nu_k^2\varphi_k(\theta).
\end{equation}
Utilizing Cheeger's separation of variables functional calculus (see  \cite[(8.47)]{Taylor} for more details), we can write the kernel of the operator $e^{-it\mathcal{H}_{\bf A}}$ as
\begin{equation}\label{equ:kernsina}
\begin{split}
K(t,x,y)&=K(t,r,\theta,r',\theta')=\sum_{k\in\Z}\varphi_{k}(\theta)\overline{\varphi_{k}(\theta')}K_{\nu_k}(t,r,r')
\end{split}
\end{equation}
where $\varphi_k(\theta)\overline{\varphi_{k}(\theta')}=\frac1{2\pi\rho}e^{-i\big( (\theta-\theta')(\frac k\rho+\alpha)-\int_{\theta'}^\theta A(\theta^*)\, d\theta^*\big)},$ and $K_{\nu}(t,r,r')$ is given by
\begin{equation}\label{equ:knukdef}
  K_{\nu}(t,r,r')=\int_0^\infty e^{-it\rho^2}J_{\nu}(r\rho)J_{\nu}(r'\rho) \,\rho\, d\rho
  =\lim_{\epsilon\searrow 0} \int_0^\infty e^{-(\epsilon+it)\rho^2}J_{\nu}(r\rho)J_{\nu}(r'\rho) \,\rho\, d\rho.
\end{equation}
However, the Weber's second exponential integral \cite[Section 13.31 (1)]{Watson} shows
\begin{equation}
\begin{split}
 \int_0^\infty  e^{-(\epsilon+it)\rho^2}  J_{\nu}(r\rho) J_{\nu}(r'\rho)\rho\,  d\rho =\frac{e^{-\frac{r^2+r'^2}{4(\epsilon+it)}}}{2(\epsilon+it)} I_\nu\big(\frac{rr'}{2(\epsilon+it)}\big),\quad \epsilon>0,
\end{split}
\end{equation}
where $I_\nu(x)$ is the modified Bessel function of the first kind. Define $z:=\frac{rr'}{2(\epsilon+it)}$, then we use the integral representation in \cite{Watson} to write
\begin{equation}\label{for-Be}
I_\nu(z)=\frac1{\pi}\int_0^\pi e^{z\cos(s)} \cos(\nu s) ds-\frac{\sin(\nu\pi)}{\pi}\int_0^\infty e^{-z\cosh s} e^{-s\nu} ds.
\end{equation}
Hence, combing \eqref{equ:kernsina}-\eqref{for-Be} together, we need to consider
\begin{equation}\label{ied-d}
\frac1{\pi}\frac{e^{-\frac{r^2+r'^2}{4(\epsilon+it)}}}{2(\epsilon+it)}\sum_{k\in\Z}\frac1{2\pi\rho}e^{-i\big( (\theta-\theta')(\frac k\rho+\alpha)-\int_{\theta'}^\theta A(\theta^*)\, d\theta^*\big)}\int_0^\pi e^{z\cos(s)} \cos(\nu s) ds
\end{equation}
and
\begin{equation}\label{ide-e}
\frac1{\pi}\frac{e^{-\frac{r^2+r'^2}{4(\epsilon+it)}}}{2(\epsilon+it)}\sum_{k\in\Z}\frac1{2\pi\rho}e^{-i\big( (\theta-\theta')(\frac k\rho+\alpha)-\int_{\theta'}^\theta A(\theta^*)\, d\theta^*\big)}\sin(\nu\pi)\int_0^\infty e^{-z\cosh s} e^{-s\nu} ds.
\end{equation}
We first consider the term \eqref{ied-d}. Based on the formula about the relation between the Dirac comb distribution and its Fourier series
\begin{equation*}
\sum_{j\in\Z} \delta(x-2\pi \rho j)=\sum_{k\in\Z} \frac1{2\pi \rho} e^{i \frac{k}{\rho}x},
\end{equation*}
and recall that $\nu_k=|\frac k\rho+\alpha|,\ k\in\Z$, we directly obtain
\begin{align*}
&\sum_{k\in\Z}\frac1{2\pi\rho}e^{-i\frac k\rho (\theta-\theta')}\cos(\nu s)\\
=&
\frac1{2\pi\rho}\sum_{k\in\Z}\frac{e^{i(\frac k\rho+\alpha)s}+e^{-i(\frac k\rho+\alpha)s}}2e^{-i\frac k\rho (\theta-\theta')}\\
=&\frac12\sum_{j\in\Z}\big[e^{i\alpha s}\delta(s-(\theta-\theta')-2j\pi \rho)
 +e^{-i\alpha s}\delta(-s-(\theta-\theta')-2j\pi\rho)\big],
\end{align*}
Therefore, we further have
\begin{align*}
&\frac1{\pi}\sum_{k\in\Z}\frac1{2\pi\rho}e^{-i\frac k\rho (\theta-\theta')}\int_0^\pi e^{z\cos(s)} \cos(\nu s) ds\\
=&\frac1{2\pi}\sum_{j\in\Z}\int_0^\pi\big[e^{i\alpha s}\delta(s-(\theta-\theta')-2j\pi \rho)
 +e^{-i\alpha s}\delta(-s-(\theta-\theta')-2j\pi\rho)\big]e^{z\cos(s)}ds\\
=&\frac1{2\pi}\sum_{\{j\in\Z: 0\leq |\theta-\theta'+2j\rho\pi|\leq \pi\}} e^{i\alpha(\theta-\theta'+2j\rho\pi)}e^{z\cos(\theta-\theta'+2j\rho\pi)}.
\end{align*}
Consequently, recall the definition of $A_{\alpha,\rho}$ in above, the contribution of \eqref{ied-d} is
\begin{equation}\label{I-g}
\frac{e^{-\frac{r^2+r'^2}{4(\epsilon+it)}} }{4\pi(\epsilon+it)}
\sum_{\{j\in\Z: 0\leq |\theta-\theta'+2j\rho\pi|\leq \pi\}}A_{\alpha,\rho}(\theta,\theta') e^{\frac{rr'}{2(\epsilon+it)}\cos(\theta-\theta'+2j\rho\pi)}.
\end{equation}

We next consider the other term \eqref{ide-e}. From $\nu_k=|\frac k\rho+\alpha|$, and $\alpha\in(-\frac1\rho,\frac1\rho)\backslash\{0\}$ with $\rho\geq1$, we conclude that
\begin{align*}
\nu_k=
\begin{cases}
\frac k\rho+\alpha,&  k\geq1,\\
|\alpha|,& k=0,\\
-(\frac k\rho+\alpha),& k\leq-1.
\end{cases}
\end{align*}
Hence we have
\begin{align*}
\sin(\pi\nu_k)=\sin\Big(\Big|\frac k\rho+\alpha\Big|\pi\Big)=
\begin{cases}
\sin\Big(\frac k\rho+\alpha\Big),&  k\geq1,\\
\sin(|\alpha|\pi),& k=0,\\
-\sin\Big(\frac k\rho+\alpha\Big),& k\leq-1.
\end{cases}
\end{align*}
Therefore we further get

\begin{align*}
&\sum_{k\in\Z}e^{-i\frac k\rho (\theta-\theta')}\sin\Big(|\nu_k|\pi\Big)e^{-s|\nu_k|}\\
=&\sum_{k\geq 1} \frac{e^{i(\frac{k}{\rho}+\alpha)\pi}-e^{-i(\frac{k}{\rho}+\alpha)\pi}}{2i} e^{-(\frac{k}{\rho}+\alpha)s}e^{-i\frac k{\rho}(\theta-\theta')}+\sin(|\alpha|\pi)e^{-|\alpha|s}\\
&-\sum_{k\leq -1} \frac{e^{i(\frac{k}{\rho}+\alpha)\pi}-e^{-i(\frac{k}{\rho}+\alpha)\pi}}{2i} e^{(\frac{k}{\rho}+\alpha)s}e^{-i\frac k{\rho}(\theta-\theta')}\\
=&\sin(|\alpha|\pi)e^{-|\alpha|s}+\frac1{2i}e^{-s\alpha}\sum_{k\geq 1} \Big(e^{i\alpha\pi}e^{i\frac{k}{\rho}(\pi-(\theta-\theta')+si)}-e^{-i\alpha\pi}e^{i\frac{k}{\rho}(-\pi-(\theta-\theta')+si)}\Big)\\
&-\frac1{2i}e^{s\alpha}\sum_{k\geq 1} \Big(e^{i\alpha\pi}e^{i\frac{k}{\rho}(-\pi+(\theta-\theta')+si)}-e^{-i\alpha\pi}e^{i\frac{k}{\rho}(\pi+(\theta-\theta')+si)}\Big).\\
\end{align*}

Notice that
\begin{equation}
\sum_{k=1}^\infty e^{ikz}=\frac{e^{iz}}{1-e^{iz}},\qquad \mathrm{Im} z>0.
\end{equation}
Keep in mind $\phi_1$ in \eqref{def-phi}, we can easily obtain that

\begin{equation*}
\sum_{k\geq1}e^{i\frac{k}{\rho}(\pi-(\theta-\theta')+si)}=\frac{e^{i\phi_1-\frac{s}{\rho}}}{1-e^{i\phi_1-\frac{s}{\rho}}}=\frac{\cos\phi_1-e^{-\frac s\rho}+i\sin\phi_1}{2(\cosh \frac s\rho-\cos\phi_1)}.
\end{equation*}

Using the same argument in above, the other terms are similarly computed. We finally obtain the contribution of \eqref{ide-e}
\begin{equation}\label{I-d}
\begin{split}
&-\frac{1}{4\pi^2\rho}\frac{e^{-\frac{r^2+r'^2}{4(\epsilon+it)}} }{(\epsilon+it)}
 \int_0^\infty e^{-\frac{rr'}{2(\epsilon+it)}\cosh s} B_{\alpha,\rho}(s,\theta,\theta') ds,
\end{split}
\end{equation}
where $B_{\alpha,\rho}(s,\theta,\theta')$ is in \eqref{def-B}. Since $\alpha\in(-\frac1\rho,\frac1\rho)\backslash\{0\}$ with $\rho\geq1$, which means the integral is converging, see \eqref{A-Bo} in below for details. Hence after taking $\epsilon\searrow0$, we obtain \eqref{pro-SK}.

\subsection{The proof of Theorem \ref{thm:Sch-dis}}  For given $\rho>0$, the summation of $j$ in the first term of \eqref{pro-SK} is finite and bounded by $O(1+\frac1{\rho})$. Thus it is easy to see
\begin{equation*}
\Big|\frac1{4\pi}\frac{e^{-\frac{r^2+r'^2}{4it}} }{it}
\sum_{\{j\in\Z: 0\leq |\theta-\theta'+2j\rho\pi|\leq \pi\}}A_{\alpha,\rho}(\theta,\theta') e^{\frac{rr'}{2it}\cos(\theta-\theta'+2j\rho\pi)}\Big|\leq C_\rho |t|^{-1}.
\end{equation*}
Next we estimate the second term in \eqref{pro-SK}. To show \eqref{dis-s}, it suffices to prove
\begin{equation}\label{A-Bo}
\begin{split}
\int_0^\infty|B_{\alpha,\rho}(s,\theta,\theta')|\;ds\leq C_\rho
\end{split},\quad \forall\alpha\in\Big(-\frac1\rho,\frac1\rho\Big)\backslash\{0\}.
\end{equation}
Note that $B_{\alpha,\rho}$ in \eqref{def-B}, it further reduced to show $\forall\alpha\in\Big(-\frac1\rho,\frac1\rho\Big)\backslash\{0\}$, there holds
\begin{align}\label{0-infty}
 &\int_0^\infty  \Big|\sin(|\alpha|s)e^{-|\alpha|s}\Big| ds\lesssim 1,\\\label{1-infty}
&\int_0^\infty\Big|\frac{e^{\pm\alpha s}\sin\phi }{\cosh\frac s\rho-\cos\phi}\Big|\;ds\lesssim_\rho 1,\\
\label{2-infty}
&\int_0^\infty\Big|\frac{e^{\pm\alpha s}(\cos\phi-e^{-\frac s\rho}) }{\cosh\frac s\rho-\cos\phi}\Big|\;ds\lesssim_\rho 1,
\end{align}
where $\phi$ denotes $\phi_1$ or $\phi_2$ for the convenience of writing.
 \vspace{0.2cm}

It is easy to check \eqref{0-infty}. Next we aim to estimate \eqref{1-infty} and \eqref{2-infty}. Note that
\begin{equation}
\cosh\frac s\rho-\cos\phi=\sinh^2\frac s{2\rho}+\sin^2\frac\phi2
\end{equation}
\emph{Estimate of \eqref{0-infty}:} Since $\alpha\in\big(-\frac1\rho,\frac1\rho\big)\backslash\{0\}$ and $\rho\geq1$ imply $\frac1\rho\pm\alpha>0$, then
\begin{align*}
&\int_0^\infty\Big|\frac{e^{\pm\alpha s}\sin\phi }{\cosh\frac s\rho-\cos\phi}\Big|\;ds\\
\lesssim&\int_0^1\Big|\frac{\sin\phi }{\sinh^2\frac s{2\rho}+\sin^2\frac\phi2}\Big|\;ds+\int_1^\infty\Big|\frac{e^{\pm\alpha s} }{\sinh^2\frac s{2\rho}+\sin^2\frac\phi2}\Big|\;ds\\
\lesssim&\int_0^1\Big|\frac{\sin\frac\phi2 }{(\frac s{\rho})^2+\sin^2\frac\phi2}\Big|\;ds+\int_1^\infty e^{-(\frac1\rho\pm\alpha) s} \;ds\\
\leq& C_\rho.
\end{align*}
\emph{Estimate of \eqref{1-infty}:} Using the fact $\frac1\rho\pm\alpha>0$ again, we get
\begin{align*}
&\int_0^\infty\Big|\frac{e^{\pm\alpha s}(\cos\phi-e^{-\frac s\rho}) }{\cosh\frac s\rho-\cos\phi}\Big|\;ds\\
\lesssim&\int_0^1\Big|\frac{(\cos\phi-1+1- e^{-\frac s\rho})e^{\pm\alpha s}}{\sinh^2\frac s{2\rho}+\sin^2\frac\phi2}\Big|\;ds+\int_1^\infty\Big|\frac{(\cos\phi-e^{-\frac s\rho})e^{\pm\alpha s}}{\sinh^2\frac s{2\rho}+\sin^2\frac\phi2}\Big|\;ds\\
\lesssim&\int_0^1\Big|\frac{s(\frac s\rho+\frac{\phi^2}2) }{(\frac s{\rho})^2+(\frac\phi2)^2}\Big|\;ds+\int_1^\infty e^{-(\frac1\rho\pm\alpha) s} \;ds\\
\leq & C_\rho.
\end{align*}
Here we finish the proof of Theorem \ref{thm:Sch-dis}, ans so Theorem \ref{thm:stri0} follows.
\vspace{0.2cm}

\section{Strichartz estimates for wave}\label{sec-w} In this section, we make use of the method in \cite{GYZZ,Zhang} to construct the spectral measure kernel and the Keel-Tao's argument in \cite{KT} to give the proof of the Theorem \ref{thm:stri-w}.

\subsection{Spectral measure kernel} The purpose of this section is to  construct the spectral measure kernel associated with $\mathcal{H}_{\bf A}$. Following the argument in \cite{GYZZ,Zhang}, the first step is to use the Schr\"odinger kernel in Theorem \ref{thm:stri-w} to obtain the resolvent kernel, which will be used to prove spectral measure according to Stone's formula.
\vspace{0.1cm}

Notice that
when $z\in \{z\in\C: \Im(z)>0\}$, there holds
$$(s-z)^{-1}=i\int_0^\infty e^{-ist} e^{iz t}dt,\quad \forall s\in\R,$$
thus we obtain, for $z=\lambda^2+i\epsilon$ with $\epsilon>0$
\begin{equation}\label{res+}
\begin{split}
(\mathcal{H}_{\bf A}-(\lambda^2+i0))^{-1}&=i\lim_{\epsilon\to 0^+}\int_0^\infty e^{-it\mathcal{H}_{\bf A}} e^{it(\lambda^2+i\epsilon)}dt.
\end{split}
\end{equation}
Recall the Schr\"odinger kernel in \eqref{pro-SK}, we need to consider
\begin{equation}\label{res-g}
i\lim_{\epsilon\rightarrow0^+}\int_0^\infty\frac1{4\pi }\sum_{\{j\in\Z: 0\leq |\theta-\theta'+2j\rho\pi|\leq \pi\}} A_{\alpha,\rho}(\theta,\theta')
\frac{e^{-\frac{r^2+r'^2}{4it}} }{it}
e^{\frac{rr'}{2it}\cos(\theta-\theta'+2j\rho\pi)}e^{it(\lambda^2+\epsilon)}\, dt.
\end{equation}
and
\begin{equation}\label{res-d}
\begin{split}
&-i\lim_{\epsilon\rightarrow0^+}\int_0^\infty\int_0^\infty\frac{1}{4\pi^2\rho}\frac{e^{-\frac{r^2+r'^2}{4it}} }{it}e^{-\frac{rr'}{2it}\cosh s}B_{\alpha,\rho}(s,\theta,\theta') \, ds \, e^{it(\lambda^2+\epsilon)}\, dt
.
\end{split}
\end{equation}
Note that for $\xi=(\xi_1,\xi_2)\in\R^2$
\begin{equation}\label{for-ee}
\int_{\R^2}e^{-ix\xi}e^{-it|\xi|^2}\,d \xi=\frac\pi{it}e^{-\frac{|x|^2}{4it}}.
\end{equation}
This implies
\begin{equation}\label{for-n1}
\frac{1 }{it}e^{-\frac{r^2+r'^2}{4it}}
 e^{\frac{rr'}{2it}\cos(\theta-\theta'+2j\rho\pi)}=\frac{1 }{it}e^{-\frac{|{\bf m}|^2}{4it}}=\frac1\pi\int_{\R^2}e^{-i{\bf m}\xi}e^{-it|\xi|^2}\,d \xi,
\end{equation}
and similarly
\begin{equation}\label{for-n2}
\frac{1}{it}e^{-\frac{r^2+r'^2}{4it}}
 e^{-\frac{rr'}{2it}\cosh s}=\frac1\pi\int_{\R^2}e^{-i{\bf  n}\xi}e^{-it|\xi|^2}\,d \xi,
\end{equation}
where \begin{equation}\label{bf-m}
{\bf m}=(r-r', \sqrt{2rr'(1-\cos(\theta-\theta'+2j\rho\pi))}),
\end{equation}
and
\begin{equation}\label{bf-n}
{\bf  n}=\big(r+r', \sqrt{2rr'(\cosh s-1)}\big).
\end{equation}
For $z=\lambda^2+i\epsilon$ with $\epsilon>0$, we use \eqref{for-n1}, \eqref{for-n2}, then \eqref{for-ee} to have
\begin{equation*}
\begin{split}
&\int_0^\infty \frac{e^{-\frac{r^2+r'^2}{4it}} }{it} e^{\frac{rr'}{2it} \cos(\theta-\theta'+2j\rho\pi)} e^{iz t}dt\\
&=\frac 1 \pi\int_{\R^2} e^{-i{\bf m}\cdot\xi}\int_0^\infty e^{-it|\xi|^2}e^{iz t}
\, dt\, d\xi=\frac {-i} \pi\int_{\R^2} \frac{e^{-i{\bf m}\cdot\xi}}{|\xi|^2-z} \, d\xi
\end{split}
\end{equation*}
and
\begin{equation*}
\begin{split}
\int_0^\infty \frac{e^{-\frac{r^2+r'^2}{4it}} }{it} e^{-\frac{rr'}{2it}\cosh s} e^{iz t}dt
=\frac {-i}\pi\int_{\R^2} \frac{e^{-i{\bf  n}\cdot\xi}}{|\xi|^2-z} d\xi.
\end{split}
\end{equation*}
Plugging these into \eqref{res-g} and \eqref{res-d} respectively, and then taking $\epsilon\to 0^+$, we finally obtain the incoming resolvent kernel $(\mathcal{H}_{\bf A}-(\lambda^2+i0))^{-1}$ written as
\begin{equation}\label{equ:res+}
\begin{split}
&(\mathcal{H}_{\bf A}-(\lambda^2+i0))^{-1}\\&=
\frac1{4\pi^2}
\sum_{\{j\in\Z: 0\leq |\theta-\theta'+2j\rho\pi|\leq \pi\}}A_{\alpha,\rho}(\theta,\theta')\int_{\R^2} \frac{e^{-i{\bf m}\cdot {\xi}}}{|\xi|^2-(\lambda^2+i0)} \, d{ \xi}\\&
-\frac{1}{4\pi^3\rho }
 \int_0^\infty \int_{\R^2} \frac{e^{-i{\bf n}\cdot {\xi}}}{|\xi|^2-(\lambda^2+i0)} \, d{ \xi} \,B_{\alpha,\rho}(s,\theta,\theta') \, ds.
 \end{split}
\end{equation}\vspace{0.1cm}

For our goal to obtain the spectral measure, we also need the outgoing resolvent kernel $(\mathcal{H}_{\bf A}-(\lambda^2-i0))^{-1}.$ Actually, one has
\begin{equation}\label{equ:res-}
\begin{split}
&(\mathcal{H}_{\bf A}-(\lambda^2-i0))^{-1}\\&=
\frac1{4\pi^2}
\sum_{\{j\in\Z: 0\leq |\theta-\theta'+2j\rho\pi|\leq \pi\}}A_{\alpha,\rho}(\theta,\theta') \int_{\R^2} \frac{e^{-i{\bf m}\cdot {\xi}}}{|\xi|^2-(\lambda^2-i0)} \, d{ \xi}\\&
-\frac{1}{4\pi^3\rho }
 \int_0^\infty \int_{\R^2} \frac{e^{-i{\bf n}\cdot {\xi}}}{|\xi|^2-(\lambda^2-i0)} \, d{ \xi} \,B_{\alpha,\rho}(s,\theta,\theta') \, ds.
 \end{split}
\end{equation}
This can be easily done from \eqref{equ:res+}. Indeed, from \eqref{def-A} and \eqref{def-B}, we see that $\overline{A}_{-\alpha,\rho}=A_{\alpha,\rho}$ and $\overline{B}_{-\alpha,\rho}=B_{\alpha,\rho}$ , which directly implies
 \begin{equation*}
\begin{split}
(\mathcal{H}_{\bf A}-(\lambda^2-i0))^{-1}=\overline{(\overline{\mathcal{H}_{\bf A}}-(\lambda^2+i0))^{-1}}.
\end{split}
\end{equation*}

In the following text, we will make use of \eqref{equ:res+} and \eqref{equ:res-} to construct the spectral measure. According to Stone's formula, the spectral measure is related to the resolvent
 \begin{equation}\label{stone}
 dE_{\sqrt{\mathcal{H}_{\bf A}}}(\lambda)=\frac{d}{d\lambda}dE_{\sqrt{\mathcal{H}_{\bf A}}}(\lambda)\,d\lambda=\frac{\lambda}{\pi i}\big(R(\lambda+i0)-R(\lambda-i0)\big)\, d\lambda
 \end{equation}
 where the resolvent $$R(\lambda\pm i0)=\lim_{\epsilon\searrow 0}(\mathcal{H}_{\bf A}-(\lambda^2\pm i\epsilon))^{-1}.$$
Apply \eqref{equ:res+} and \eqref{equ:res-} to \eqref{stone}, it follows that
 \begin{equation*}
 \begin{split}
 &dE_{\sqrt{\mathcal{H}_{\bf A}}}(\lambda;x,y)\\
& =
\frac1{4\pi} \frac{\lambda}{\pi^2 i}
\sum_{\{j\in\Z: 0\leq |\theta-\theta'+2j\rho\pi|\leq \pi\}}A_{\alpha,\rho}(\theta,\theta') \int_{\R^2} e^{-i{\bf m}\cdot {\xi}}\Big(\frac1{|\xi|^2-(\lambda^2+i0)}-\frac1{|\xi|^2-(\lambda^2-i0)}\Big) \, d{ \xi}
 \\&
-\frac{1}{4\pi^2 \rho} \frac{\lambda}{\pi^2 i}
 \int_0^\infty \int_{\R^2} e^{-i{\bf n}\cdot {\xi}}\Big(\frac1{|\xi|^2-(\lambda^2+i0)}-\frac1{|\xi|^2-(\lambda^2-i0)}\Big) \, d{ \xi}
B_{\alpha,\rho}(s,\theta,\theta') ds.
\end{split}
\end{equation*}
On the one hand, we note that
\begin{equation*}\label{id-spect}
\begin{split}
&\lim_{\epsilon\to 0^+}\frac{\lambda}{\pi i}\int_{\R^2} e^{-ix\cdot\xi}\Big(\frac{1}{|\xi|^2-(\lambda^2+i\epsilon)}-\frac{1}{|\xi|^2-(\lambda^2-i\epsilon)}\Big) d\xi\\
&=\lim_{\epsilon\to 0^+} \frac{\lambda}{\pi }\int_{\R^2} e^{-ix\cdot\xi}2\Im\Big(\frac{1}{|\xi|^2-(\lambda^2+i\epsilon)}\Big)d\xi\\
&=\lim_{\epsilon\to 0^+} \frac{\lambda}{\pi }\int_{0}^\infty \frac{2\epsilon}{(\rho^2-\lambda^2)^2+\epsilon^2} \int_{|\omega|=1} e^{-i\rho x\cdot\omega} d\rho_\omega  \, \rho d\rho\\
&= \lambda \int_{|\omega|=1} e^{-i\lambda x\cdot\omega} d\rho_\omega  \\
\end{split}
\end{equation*}
where we use the fact
 the  Poisson kernel is is an approximation to the identity which implies that, for any reasonable function $m(x)$
\begin{equation*}
\begin{split}
m(x)&=\lim_{\epsilon\to 0^+}\frac1\pi \int_{\R} \Im\big(\frac{1}{x-(y+i\epsilon)}\big) m(y)dy
\\&=\lim_{\epsilon\to 0^+}\frac1\pi \int_{\R} \frac{\epsilon}{(x-y)^2+\epsilon^2} m(y)dy.
\end{split}
\end{equation*}
On the other hand, for example \cite[Theorem 1.2.1]{sogge}, we also observe that
\begin{equation*}
\begin{split}
\int_{\mathbb{S}^{1}} e^{-i x\cdot\omega} d\rho_\omega=\sum_{\pm}  a_\pm(|x|) e^{\pm i|x|}
\end{split}
\end{equation*}
where
\begin{equation}\label{est-a}
\begin{split}
| \partial_r^k a_\pm(r)|\leq C_k(1+r)^{-\frac{1}2-k},\quad k\geq 0.
\end{split}
\end{equation}
Therefore we obtain that
 \begin{equation}\label{sp}
 \begin{split}
 dE_{\sqrt{\mathcal{H}_{\bf A}}}(\lambda;x,y)& =
\frac{\lambda}{4\pi^2} \sum_{\pm}\Big(
\sum_{\{j\in\Z: 0\leq |\theta-\theta'+2j\rho\pi|\leq \pi\}}  A_{\alpha,\rho}(\theta,\theta')a_\pm(\lambda |{\bf{m}}|)e^{\pm i\lambda |{\bf m}|}
 \\&
-\frac{1}{ 4\pi^3 \rho}\int_0^\infty a_\pm(\lambda |{\bf{ n}}|)e^{\pm i\lambda |{\bf  n}|}
B_{\alpha,\rho}(s,\theta,\theta') ds\Big),
\end{split}
\end{equation}
where $|{\bf m}|=\sqrt{r^2+r'^2-2rr'\cos(\theta-\theta'+2j\rho\pi)}$, $|{\bf n}|=\sqrt{r^2+r'^2+2rr'\cosh s}$, and
$A_{\alpha,\rho}(\theta,\theta')$ and $B_{\alpha,\rho}(s,\theta,\theta')$ is in \eqref{def-A} and \eqref{def-B} respectively

\subsection{Strichartz estimates for Wave} In this subsection, we aim to prove Theorem \ref{thm:stri-w}. The main ingredients are the Littlewood-Paley square function inequality associated with the operator $\mathcal{H}_{\A}$ and the following dispersive estimates of $e^{it\sqrt{\mathcal{H}_{A}}}.$
\begin{theorem}[Dispersive estimates for wave]\label{thm:dis-w} Let $\mathcal{H}_{A}$ be the magnetic operator on $\Sigma$ and
let $\phi\in C_c^\infty([1/2, 2])$ be in \eqref{def-phi}. Assume $f=\phi(2^{-k}\sqrt{\mathcal{H}_{A}})f$ with $k\in \Z$, then there exists a constant $C$ independent of $t$ and $k\in \Z$ such that
\begin{equation}\label{dispersive}
\begin{split}
\|e^{it\sqrt{\mathcal{L}_{A}}} f\|_{L^\infty(\Sigma)}\leq C 2^{\frac32 k}(2^{-k}+|t|)^{-\frac12}\|f\|_{L^1(\Sigma)}.
\end{split}
\end{equation}
\end{theorem}
\begin{proof} We use spectral measure kernel \eqref{sp} and stationary phase argument to obtain the dispersive estimates \eqref{dispersive}. We begin with writting
 \begin{equation*}
e^{it\sqrt{\mathcal{H}_{A}}}f=\int_\Sigma \int_0^\infty e^{it\lambda} \phi(2^{-k}\lambda) dE_{\sqrt{\mathcal{H}_{A}}}(\lambda; x, y) f(y) dy.
  \end{equation*}
  Then it suffices to show kernel estimate
  \begin{equation}\label{est:k}
\Big|  \int_0^\infty e^{it\lambda} \phi(2^{-k}\lambda) dE_{\sqrt{\mathcal{H}_{\bf A}}}(\lambda; x, y) \Big|\leq C 2^{\frac32 k}(2^{-k}+|t|)^{-\frac12}.
\end{equation}
Recall the fact that the summation of $j$ in the first term is finite, and so for given $\rho>0$, one has
$$\Big|\sum_{\{j\in\Z: 0\leq |\theta-\theta'+2j\rho\pi|\leq \pi\}}  A_{\alpha,\rho}(\theta,\theta')\Big|\lesssim O(1+\frac1{\rho}).$$
Thus to finish \eqref{dispersive}, from spectral measure \eqref{sp}, we aim to estimate
  \begin{equation}\label{est:k1}
\Big| \int_0^\infty e^{it\lambda} \phi(2^{-k}\lambda)\lambda a_\pm(\lambda |{\bf m}|)e^{\pm i\lambda |{\bf m}|} d\lambda\Big|\leq C 2^{\frac32 k}(2^{-k}+|t|)^{-\frac12},
\end{equation}
and
  \begin{equation}\label{est:k2}
  \begin{split}
\Big|  \int_0^\infty e^{it\lambda} \phi(2^{-k}\lambda)\lambda \int_0^\infty a_\pm(\lambda |{\bf n}|)e^{\pm i\lambda |{\bf n}|}
&B_{\alpha,\rho}(s,\theta,\theta') ds \,d\lambda\Big|\\&\leq C 2^{\frac32 k}(2^{-k}+|t|)^{-\frac12},
\end{split}
\end{equation}
where $a_\pm$ satisfies \eqref{est-a},
and ${\bf{m}}, {\bf{n}}$ are in \eqref{bf-m} and \eqref{bf-n}, and $B_{\alpha,\rho}(s,\theta,\theta')$ is in \eqref{def-B}.
For the convience of notations, let $d=|{\bf m}|$ or $|{\bf n}|$. Since $a_\pm$ satisfies \eqref{est-a}, we then conclude
\begin{equation}\label{bean'}
 |\partial_\lambda^N [a_\pm(\lambda d)]|\leq C_N \lambda^{-N}(1+\lambda d)^{-\frac{1}2},\quad N\geq 0.
\end{equation}
We first prove \eqref{est:k1}. By \eqref{bean'}, we use the $N$-times integration by parts to obtain
\begin{equation*}
\begin{split}
&\Big|  \int_0^\infty e^{it\lambda} \phi(2^{-k}\lambda)\lambda a_\pm(\lambda d)e^{\pm i\lambda d} d\lambda\Big|\\&\leq \Big|\int_0^\infty \left(\frac1{
i(t\pm d)}\frac\partial{\partial\lambda}\right)^{N}\big(e^{i(t\pm d)\lambda}\big)
\phi(2^{-k}\lambda)\lambda a_\pm(\lambda d) d\lambda\Big|\\& \leq
C_N|t\pm d|^{-N}\int_{2^{k-1}}^{2^{k+1}}\lambda^{1-N}(1+\lambda
d)^{-1/2}d\lambda\\&\leq
C_N2^{k(2-N)}|t\pm d|^{-N}(1+2^kd)^{-1/2}.
\end{split}
\end{equation*}
It follows that
\begin{equation}\label{dispersive2}
\begin{split}
&\Big|  \int_0^\infty e^{it\lambda} \phi(2^{-k}\lambda)\lambda a_\pm(\lambda d)e^{\pm i\lambda d} d\lambda\Big|\\&\leq
C_N2^{2k}\big(1+2^k|t\pm d|\big)^{-N}(1+2^k d)^{-1/2}.
\end{split}
\end{equation}
If $|t|\sim d$, we see \eqref{est:k1}.
Otherwise, we have $|t\pm d|\geq c|t|$ for some small constant
$c$, choose $N=1$ and $N=0$,  and then use geometric mean argument to prove \eqref{est:k1}.\vspace{0.2cm}

We next prove \eqref{est:k2}. We follow the same lines to obtain
  \begin{equation}\label{est-k2}
  \begin{split}
\Big|  \int_0^\infty e^{it\lambda} \phi(2^{-k}\lambda)\lambda &\int_0^\infty a_\pm(\lambda d)e^{\pm i\lambda |d|}
B_{\alpha,\rho}(s,\theta,\theta') ds \,d\lambda\Big|\\&\leq C 2^{\frac32 k}(2^{-k}+|t|)^{-\frac12} \int_0^\infty |B_{\alpha,\rho}(s,\theta,\theta')| ds.
\end{split}
\end{equation}
Since \eqref{A-Bo} holds, which implies \eqref{est:k2}, hence \eqref{dispersive}.
\end{proof}

As a consequence of Theorem \ref{thm:dis-w}, we directly have
\begin{proposition}\label{prop-w}
Let $U_k(t)=e^{it\sqrt{\mathcal{H}_{A}}}\phi(2^{-k}\sqrt{\mathcal{H}_{A}})$ be defined by
\begin{equation*}
U_k(t)f:=\int_\Sigma \int_0^\infty e^{it\lambda} \phi(2^{-k}\lambda) dE_{\sqrt{\mathcal{H}_{A}}}(\lambda; x, y) f(y) dy.
  \end{equation*}
Then there exists a constant $C$ independent of $t,x,y$ for all $k\in\Z$ such that
\begin{equation}\label{est-lo}
\|U_k(t)U_k^*(s) f\|_{L^\infty(\Sigma)}\leq C 2^{\frac32 k}(2^{-k}+|t-s|)^{-\frac12}\|f\|_{L^1(\Sigma)}.
\end{equation}
\end{proposition}
\vspace{0.2cm}
As above stated, to prove Strichartz estimates \eqref{stri-w}, we need to rebuild the Littlewood-Paley square function inequality with $\mathcal{H}_{\bf A}$ on the flat cone.
\begin{proposition}[LP square function inequality]\label{prop:squarefun} Let $\{\phi_k\}_{k\in\mathbb Z}$ be in \eqref{def-phi} and let $\mathcal{H}_{\A}$ be given in \eqref{oper}.
Then for $1<p<\infty$,
there exist constants $c_p$ and $C_p$ depending on $p$ such that
\begin{equation}\label{square}
c_p\|f\|_{L^p(\Sigma)}\leq
\Big\|\Big(\sum_{k\in\Z}|\phi_k(\sqrt{\LL_{{\A}}})f|^2\Big)^{\frac12}\Big\|_{L^p(\Sigma)}\leq
C_p\|f\|_{L^p(\Sigma)}.
\end{equation}
\end{proposition}
\begin{proof}
The result depends on the Gaussian boundedness of the heat kernel of $\mathcal{H}_{\A}$,
\begin{equation}\label{est-ker}
|e^{-t\mathcal{H}_{\A}}(x,y)|\lesssim\frac1te^{-\frac{|x-y|^2}{4t}},\quad t>0.
\end{equation}
Indeed, the inequality can be easily obtained following from the argument in \cite{Li00}, in which Li has proved \eqref{est-ker} for the Laplace operator $\Delta_g$ on our setting.

The kernel estimate \eqref{est-ker} implies, in a standard way (see Alexopoulos \cite[Theorem 6.1]{A04} or \cite{KMMZZ18}) that
\begin{equation}\label{M-B}
m(\sqrt{\mathcal{H}_{\A}}):L^p(\Sigma)\rightarrow L^p(\Sigma),\quad1<p<\infty,
\end{equation}
where $m\in C^N(X)$ satisfies the weaker Mikhlin-type condition for $N\geq2$
\begin{equation*}
\sup_{0\leq k\leq N}\sup_{\lambda\in\R}|(\lambda\partial_\lambda)^km(\lambda)|<\infty.
\end{equation*}
Based on \eqref{M-B}, we can get \eqref{square} by utilizing Stein's \cite{Stein70} classical argument.
\end{proof}
\vspace{0.2cm}

We now turn to show our mian Theorem \ref{thm:stri-w}. We start by introducing the Keel-Tao's theorem.
\begin{theorem}[Keel and Tao \cite{KT}]\label{KT}
Let $(X,\mathcal{M},\mu)$ be a $\gamma-$finite
mesure space and let $U:\mathbb{R}\rightarrow B(L^2(X,\mathcal{M},\mu))$ be a weakly measurable map satisfying, for some constant $C,\beta\geq0$, $\gamma,h>0$,
\begin{equation}\label{est-w}
\begin{split}
\|U(t)\|_{L^2\rightarrow L^2}\leq C,\quad t\in\R,\\
\|U(t)U^*(s)\|_{L^\infty}\leq Ch^{-\beta}(h+|t-s|)^{-\gamma}\|f\|_{L^1}.
\end{split}
\end{equation}
Then for every pair $(q,r)\in[1,\infty]$ such that $(q,r,\gamma)\ne(2,\infty,1)$ and
\begin{equation}
\frac 1q+\frac \gamma r\leq\frac\gamma 2,\quad q\geq2,
\end{equation}
there exists a constant $\widetilde C$ only depending on $C,\gamma, q$ and $r$ such that
\begin{equation}
\Big(\int_\R\|U(t)u_0\|_{L^r}^q\Big)^{\frac1q}\leq\widetilde C\Lambda(h)\|u_0\|_{L^2},
\end{equation}
where $\Lambda(h)=h^{-(\beta+\gamma)(\frac12-\frac1r)+\frac1q}.$
\end{theorem}

Denote $U(t):=e^{it\sqrt{\mathcal{H}_{\A}}}$. Thus we can write the solution
\begin{equation}\label{4.19}
u(t,x)=\frac{U(t)+U(-t)}2
f(x)+\frac{U(t)-U(-t)}{2i\sqrt{
\mathcal{L}_{\A}}}g(x),
\end{equation}
For each $k\in \mathbb{Z}$, let $\phi_k(\lambda)=\phi(2^{-k}\lambda)$, we define
\begin{align}\label{propa}
U_k(t)=\int^\infty_0e^{it\lambda} \phi_k(\lambda) dE_{\sqrt{\mathcal{H}_{\A}}}(\lambda;x,y),
\end{align}
where $dE_{\sqrt{\mathcal{H}_{\A}}}(\lambda;x,y)$ is the spectral measure in \eqref{sp}. Then
\begin{equation}\label{d-c}
U(t)f(x)=\sum_{k\in\Z} U_k(t)f(x)=\sum_{j\in\Z}\sum_{k\in\Z} U_k(t)f_j(x),
\end{equation}
where $f_j(x)=\phi(2^{-j}\sqrt{\mathcal{H}_{\A}})f$. By using the Littlewood-Paley square function inequality \eqref{square} and Minkowski's inequality, we show that
\begin{align}\label{4.21}
\|U(t)f\|_{L_{t}^{q}(\mathbb{R}L^r(\Sigma))}
\lesssim&\Big\|\big(\sum_{j\in\mathbb{Z}}|\sum_{k\in\mathbb{Z}}U_k(t)f_j|^2\big)^{1/2}\Big\|_{L_{t}^{q}(\mathbb{R}L^r(\Sigma))}
\nonumber\\
\lesssim&\Big(\sum_{j\in\mathbb{Z}}\big\|\sum_{k\in\mathbb{Z}}U_k(t)f_j\big\|_{L_{t}^{q}(\mathbb{R}L^r(\Sigma))}
^2\Big)^{1/2}.
\end{align}\label{4.22}
The energy estimate, and the localized dispersive estimate  \eqref{est-lo} imply that $U_k(t)$ satisfies estimates \eqref{est-w} with $h=2^{-k}$, $\beta=3/2$ and $\gamma=1/2$. Then it follows from that for $(q,r)\in\Lambda^w_s$, the localized Strichartz estimates holds
\begin{equation}\label{est-wSt}
\|U_k(t)f\|_{L_{t}^{q}(\mathbb{R}L^r(\Sigma))}\lesssim2^{k[2(\frac12-\frac1r)-\frac1q]}\|f\|_{L^2(\Sigma)}=2^{ks}\|f\|_{L^2(\Sigma)},
\end{equation}
On account of $f_j=\phi_j(\sqrt{\mathcal{H}_{\A}})f$, then $\phi_k(\sqrt{\mathcal{H}_{\A}})f_j$ vanishes when $|j-k|\geq 10$. Hence
\begin{align}\label{4.23}
\Big(\sum_{j\in\mathbb{Z}}\big\|\sum_{k\in\mathbb{Z}}U_k(t)f_j\big\|_{L_{t}^{q}(\mathbb{R}L^r(\Sigma))}
^2\Big)^{1/2}\lesssim&\Big(\sum_{j\in\mathbb{Z}}\sum_{|j-k|\leq10}\|U_k(t)f_j\|_{L_{t}^{q}(\mathbb{R}L^r(\Sigma))}
^2\Big)^{1/2}\nonumber\\
\lesssim&\Big(\sum_{j\in\mathbb{Z}}2^{2js}\|f_j\|_{L^{2}(\Sigma)}
^2\Big)^{1/2}=\|f\|_{\dot{H}^s_{\bf A}(\Sigma)},
\end{align}
where we have used \eqref{est-wSt}. Here we finish the proof of Theorem \ref{thm:stri-w}.\vspace{0.2cm}

{\bf Acknowledgements}

\begin{center}

\end{center}

\end{document}